\documentclass[12pt]{amsart}
\usepackage{amssymb,latexsym, amscd, a4wide}
\usepackage[all]{xy}
\usepackage{graphicx}
\usepackage{mathrsfs}
\usepackage{amsmath}
\usepackage{pb-diagram}
\usepackage{enumerate}

 \newtheorem*{cor}{Corollary}
\newtheorem*{Cor}{Corollary}

 \newtheorem*{proposition}{Proposition} \theoremstyle{definition}
  \theoremstyle{definition}
 \theoremstyle{definition}
 \theoremstyle{definition}

 \newtheorem*{rem}{Remark}
\newtheorem{rem*}{Remark}
 \numberwithin{equation}{section}
 \newcommand{\nc}{\newcommand}
\nc{\ssn}{\subsection{}} \nc{\sssn}{\subsubsection{}}

\newcommand{\Ho}{\operatorname{H}}

\newcommand{\tU}{{\widetilde{\U}_q}}
\newcommand{\tW}{{\tilde{\W}}}

\newcommand{\Ind}{{\operatornamewithlimits{Ind}}}

\newcommand{\Indd}{{\operatornamewithlimits{Ind}\,}}
\newcommand{\res}{\operatorname{res}}
\newcommand{\Ker}{\operatorname{Ker}}
\newcommand{\Coker}{\operatorname{Coker}}

\newcommand{\Hom}{\operatorname{Hom}}
\newcommand{\End}{\operatorname{End}}
\newcommand{\Ext}{\operatorname{Ext}}
\newcommand{\Mat}{\operatorname{Mat}}
\newcommand{\Ann}{\operatorname{Ann}}
\newcommand{\Aut}{\operatorname{Aut}}
\newcommand{\Jm}{\operatorname{Im}}

\newcommand{\rank}{\operatorname{rank}}

\newcommand{\gr}{\operatorname{gr}}

\newcommand{\mood}{\hbox{\ensuremath{
\operatorname{mod}}}}
\newcommand{\mGq}{\hbox{\ensuremath{
\operatorname{Mod}(G_q)}}}
\newcommand{\mBq}{\hbox{\ensuremath{
\operatorname{Mod}(B_q)}}}

\newcommand{\MOD}{\ensuremath{\text{-mod}}}

\nc{\dirlim}{\underset{\rightarrow}{\operatorname{lim}}}
\nc{\invlim}{\underset{\leftarrow}{\operatorname{lim}}}

\newcommand{\IndBG}{\ensuremath{\Ind^{G_q}_{B_q}}}

\newcommand{\I}{\ensuremath{\operatorname{I}}}
\newcommand{\Z}{\ensuremath{\operatorname{Z}}}
\newcommand{\Zl}{\ensuremath{\operatorname{Z}^{(\ell)}}}
\newcommand{\ZHC}{\ensuremath{\operatorname{Z}^{HC}}}

\newcommand{\isoto}{\ensuremath{\overset{\sim}{\longrightarrow}}}

\newcommand{\C}{\ensuremath{{\mathbb{C}}}}

\newcommand{\AL}{\ensuremath{{\mathcal{A}\Lambda}}}
\newcommand{\D}{\ensuremath{\mathcal{D}}}

\newcommand{\Dq}{\ensuremath{\mathcal{D}_q}}

\newcommand{\pql}{\ensuremath{\phi^\lambda_{q}}}
\newcommand{\pAl}{\ensuremath{\phi^\lambda_{\A}}}

\newcommand{\m}{\ensuremath{\mathfrak{m}}}

\newcommand{\fin}{\ensuremath{{\hbox{\tiny{int}}}}}
\newcommand{\g}{\ensuremath{\mathfrak{g}}}
\newcommand{\bo}{\ensuremath{\mathfrak{b}}}
\newcommand{\bobar}{\ensuremath{\overline{\mathfrak{b}}}}
\newcommand{\n}{\ensuremath{\mathfrak{n}}}
\newcommand{\nbar}{\ensuremath{\overline{\mathfrak{n}}}}

\newcommand{\h}{\ensuremath{\mathfrak{h}}}

\newcommand{\BGG}{\ensuremath{\mathcal{O}}}

\newcommand{\Mq}{\ensuremath{\wt{M}_{q}}}
\newcommand{\MA}{\ensuremath{\wt{M}_{\A}}}
\newcommand{\MAf}{\ensuremath{\wt{M}_{\A_f}}}
\newcommand{\MAl}{\ensuremath{{M}_{\A}{(\lambda)}}}
\newcommand{\MAfl}{\ensuremath{{M}_{\A_f}{(\lambda)}}}
\newcommand{\Mql}{\ensuremath{{M}_{q}{(\lambda)}}}

\newcommand{\mSpec}{\ensuremath{\operatorname{Maxspec}}}

\newcommand{\UA}{\ensuremath{\U_\mathcal{A}}}
\newcommand{\UAl}{\ensuremath{\U^\lambda_\mathcal{A}}}
\newcommand{\Uonel}{\ensuremath{\U^\lambda_1}}

\newcommand{\tUA}{\ensuremath{\widetilde{\U}_\mathcal{A}}}

\newcommand{\UAf}{\ensuremath{\U^\fin_\mathcal{A}}}
\newcommand{\Uonef}{\ensuremath{\U^\fin_1}}

\newcommand{\UAres}{\ensuremath{\U^{\res}_\mathcal{A}}}

\newcommand{\A}{\ensuremath{\mathcal{A}}}
\newcommand{\OA}{\ensuremath{\mathcal{O}_\mathcal{A}}}

\newcommand{\W}{\ensuremath{{\mathcal{W}}}}

\newcommand{\U}{\ensuremath{\operatorname{U}}}

\newcommand{\Uq}{\ensuremath{\operatorname{U}_q}}
\newcommand{\Uql}{\ensuremath{\operatorname{U}^\lambda_{q}}}

\newcommand{\Uqf}{\ensuremath{\operatorname{U}^{{\fin}}_q}}

\newcommand{\Dql}{\ensuremath{\D^\lambda_{q}}}
\newcommand{\Uqres}{\ensuremath{\U^{\res}_q}}

\newcommand{\DBGl}{\ensuremath{\operatorname{Mod}(\Dq,B_q,\Uq(\n), \lambda)}}

\newcommand{\wt}{\ensuremath{\widetilde}}

\newcommand{\ot}{\ensuremath{\otimes}}
\newcommand{\uq}{\ensuremath{\mathfrak{u}_q}}
\newcommand{\bq}{\ensuremath{\mathfrak{u}_q}(\bo)}

\newcommand{\Oq}{\ensuremath{\mathcal{O}_q}}

\nc{\til}{\widetilde} \nc{\oppA}{A^{\sharp}} \nc{\Lemma}{{\bf Lemma.\
}} \nc{\Theorem}{{\bf Theorem.}\ } \nc{\Coro}{{\bf Corollary.}\ }
\nc{\Def}{{\bf Definition.}\ } \nc{\Prop}{{\bf Proposition.}\ }
\nc{\Con}{{\bf Conjecture.}\ } \nc{\Rem}{{\bf Remark.}\ }
\nc{\dok}{\noindent{\bf Proof.}\ }
\begin{document}
\title[Global quantum differential operators]{Global quantum differential operators on quantum flag manifolds, theorems of Duflo and
Kostant.}
\address{Erik Backelin, Departamento de Matem\'{a}ticas, Universidad de los Andes,
Carrera 1 N. 18A - 10, Bogot\'a, Colombia}
\email{erbackel@uniandes.edu.co}
\address{Kobi Kremnitzer, Mathematical Institute, University of Oxford, 24ￄ1�7 St Giles'
Oxford OX1 3LB, UK}
 \email{kremnitzer@maths.ox.ac.uk}
\date{}
\author{Erik Backelin and Kobi Kremnizer}
\begin{abstract} We give a new proof for the theorem that global sections of the sheaf of
quantum differential operators on a quantum flag manifold are given by the quantum group and that its derived global sections vanish. As corollaries we retrieve Joseph and Letzter's quantum versions of classical
enveloping algebra theorems of Duflo and Kostant. We also describe the center of the ad-integrable part of the  quantum group and the adjoint Lie algebra action on it.
\end{abstract}
\maketitle

\section{Introduction.}

\subsection{Summary.}
\subsubsection{} In \cite{BK06,BK08, BK11} we developed a localization theory for quantum groups,
that is,  a theory of quantum $\mathcal{D}$-modules on quantum
flag manifolds, for the purpose of attaining a better
understanding of their representation theory. A key result was the
computation of the global sections of the sheaf of quantum
differential operators and the vanishing of its higher
cohomologies.

Here we give new proofs for these facts, Theorem \ref{correct global sections root of unity}, Corollary \ref{correct
global sections} and Theorem \ref{higher vanishing of Verma}, that
we hope are conceptually clear and reasonably simple. We first establish the results at roots of unity. 
In this case our quantum objects are free and generically Azumaya over their centers, which are flag manifolds, their cotangent bundles and related algebro-geometric objects. 
We use known facts about these and about their sheaves of classical (non-quantum) differential operators to derive our result.

We then extend to other $q$ (more precisely, to all generic $q$ except perhaps a finite number of algebraic numbers) by means of an integral form. This is very natural as roots
of unity are Zariski dense.

\subsubsection{}
The computation of global sections in \cite{BK06} had some gaps at a root of unity. For instance, the argument in step i) in the proof of Proposition 4.8 is flawed at a root of unity since the primitive quotient of the quantum group doesn't act faithfully on a Verma module then.  Moreover,
our computation was largely based on the important but difficult papers \cite{JL92, JL94} of Joseph and Letzter.
These papers are not formulated in terms of integral forms and it is unclear whether their arguments - and hence our computation - really work at a root of unity.

On the other hand, the present
paper is (besides, of course, using and being inspired by some of their ideas) independent of their work. In fact, as corollaries
of Theorem \ref{correct global sections} we reprove - and at a
root of unity give the first complete proof for - their main
results: A quantum version of Kostant's separation of variables
theorem, \cite{K63}, and a quantum Duflo formula for Verma module
annihilators. These basic structure results for the quantum group
occur in quantum localization theory just as their classical
counterparts do in Beilinson-Bernstein localization theory for Lie
algebras, \cite{BB81}.

It is also worth mentioning that we here, Theorem \ref{higher vanishing of Verma}, establish vanishing of higher global sections at all odd roots of unity ($\geq$ Coxeter number). In \cite{BK08}
we only established this for roots of unity of prime order by specializing to the modular case of \cite{BMR}.

\smallskip

We also believe that our method will have other interesting
applications; such as in studying quantized
multiplicative quiver varieties \cite{J10}.

\subsubsection{} Another theme that we pursue is to describe the center of the \textit{ad-integrable} part of
the quantum group (i.e. the maximal subalgebra on which the
adjoint action is integrable) in terms of the algebra of functions
on a semi-simple group,  Section \ref{centers}. Our results here
extend those of \cite{CKP92} for the usual quantum group.  We
shall use them in a future paper about localization theory at a
root of unity, where it is natural to consider both the
(ad-integrable) quantum group and quantum differential operators
as sheaves over their centers.

\subsection{Preliminaries.}
\subsubsection{} We shall use the notations of  \cite{BK11}, where more
detailed background material can also be found. Let $\g$ be a
semi-simple Lie algebra,  $c_\g$ the maximal Coxeter number of the
simple factors of $\g$ and fix a Cartan subalgebra $\h \subset
\g$. In fact, semi-simplicity is assumed just to simplify
notations; our results do generalize to a reductive $\g$ with minor
modifications.

We have the DeConcini-Kac quantum group $\Uq := \Uq(\g)$ and
Lusztig's integral form $\Uqres := \Uqres(\g)$ (that contains
divided powers).  Let $\Lambda_r \subset \Lambda \subset \h^*$ be the root
lattice contained in the weight lattice. We use simply connected versions; thus the toral
part $\U^0_q$ of $\Uq$ is the group algebra $\C \Lambda$. Let
$\Lambda \ni \gamma \mapsto K_{\gamma} \in \C\Lambda$ denote the
canonical embedding. Let $T_\Lambda := \mSpec  \C\Lambda =
\Hom_{\C\text{-}alg}(\C\Lambda, \C)$.

Let $\n \subset \bo$ be a Borel subalgebra of $\g$ and its
unipotent radical, so that $\bo = \h \oplus \n$, and let
$\Uq(\bo)$ and $\Uq(\n)$ be their quantizations. $\bobar$ is the
opposite Borel and $\nbar$ its unipotent radical. We let $\Oq :=
\Oq(G)$ be the finite dual of $\Uqres$. Let $\Oq(B)$ be the dual
of $\Uq(\bo)$ which is a quotient Hopf algebra of $\Oq$. $\Oq(T)$ is the dual of $\Uq(\h)$.  $\Oq$
comes with a natural right action of $\Uqres$ which integrates to
a $G_q$-action. We abbreviate right $G_q$-action for a left
$\Oq(G)$-coaction. By $G_q$-modules we mean right $G_q$-modules and denote by $\mGq$ the category they form; same thing with $B_q$ and $T_q$. For $V \in \operatorname{Mod}(T_q)$ we denote by $\Lambda(V) \subseteq \Lambda$ its set of weights.

\subsubsection{}
Let  $\Mq := \Uq \ot_{\Uq(\n)} \C$ be a universal Verma module for
$\Uq$. For $\lambda \in T_\Lambda$ there is the corresponding
one-dimensional representation $\C_\lambda$ of $\Uq(\bo)$ and the
Verma module $\Mql := \Uq \ot_{\Uq(\bo)} \C_\lambda = \Mq
\ot_{\C\Lambda} \C_\lambda$. They have obvious left $\Uq$-actions.
Consider the right adjoint action of $\Uqres(\bo)$ on $\Uq$. It induces
integrable $\Uqres(\bo)$-actions on the quotients $\Mq$ and
$\Mql$.  Thus, $\Mq, \Mql \in \mBq$.  

Let $(\Lambda_r)_+ \subset \Lambda_r$ be the
semi-group generated by the positive roots and $\Lambda_{+} \subset \Lambda$ the
positive weights. By convention we have here chosen the positive roots such that $\Lambda(\Mq) = \Lambda(\Mql) = - \Lambda_r$. 
\emph{Note in particular that $M_q(\lambda)$ has highest weight $0$ with respect
to this right adjoint $B_q$-action!}

\subsubsection{}
Let $t$ be a variable and let $\A$ be the algebra $\C[t,t^{-1}]$
localized at  all $t-q$, where $q \neq 1$ runs over the roots of
unity of order $\leq c_\g$ and over all even roots of unity. 
In general, $q$ will denote a non-zero complex number. We shall use the
dictionary ``all $q$" = ``all $q \in \C^*$ except roots of unity of
even order or of order $\leq c_\g$", ``generic $q$" = ``all $q \in \C^*$
except roots of unity". By convention $1$ is not a root of unity
(thus $1$ is generic). From now on roots of unity are supposed to
be primitive of odd order $\geq c_\g$.

Let $\C_q := \A/(t-q)$, so
that specialization $t \mapsto q$ is given by the functor $( \ )_q
:= ( \ ) \ot_\A \C_q$. All the $q$-forms previously introduced
admit natural  $\A$-forms:  $\UA, \UA^{\res}, \UA(\bo), \OA$,
$\MA$, etc. We put $T_{\Lambda, \A} :=
\Hom_{\A\text{-}alg}(\A\Lambda, \A)$. We shall consider
$T_\Lambda$ as a subset of $T_{\Lambda, \A}$ by means of the
inclusion $\C \hookrightarrow \A$. This way we get the $\A$-form
$\MAl := \MA \ot_{\AL} \A_\lambda$, for $\lambda \in T_\Lambda$.

\subsubsection{}
Let $\W$ be the Weyl group, let $\Delta$ be the simple roots and
let $\omega_\alpha$ be the fundamental weight corresponding to
$\alpha \in \Delta$. Let $\Theta \subset \Aut \C\Lambda$ be the
group generated by the maps $$\tau_\alpha: K_{\omega_\beta}
\mapsto (-1)^{\delta_{\alpha,\beta}}K_{\omega_\beta},$$ for
$\alpha, \beta \in \Delta$. Let $\tW = \Theta \rtimes \W$ be the
extended Weyl group.  $\tW$ acts on $\C\Lambda$ and we have
$\C\Lambda^{\Theta} = \C 2\Lambda$ and $\C\Lambda^{\tW} = \C
2\Lambda^\W$. Here, and always, the $\W$-invariants are taken with
respect to the $\bullet$-action. Let $\ZHC$ be the Harish-Chandra
center of $\Uq$ and $\chi: \ZHC \isoto \C\Lambda^{\tW}$ the
Harish-Chandra isomorphism. This has an $\A$-form $\chi: \ZHC_\A := \Z(\UA) \isoto \A\Lambda^{\tW}$.

\subsubsection{} \label{commuting some stuff}
For any (say right) $\UAres$-module $M$ we put $$M^\fin := \{m \in M; ad_{\UAres}(m) \hbox{ is f.g. over } \A\}.$$
Then the $\UAres$-action on $M^\fin$ integrates to a $G_\A$-action. Similarly, for a $\Uqres$-module $N$ we put
$N^\fin := \{n \in N; \dim ad_{\Uqres}(n) < \infty\}$. Then we have $(M^\fin)_q = (M_q)^\fin$, for $q \in \C^*$.  Here and always $\dim = \dim_\C$.

One checks that $\UAf$ is a subalgebra of $\UA$. We put $\UAl := \UAf/(\Ker \chi_\lambda)$ and $\Uql :=
\Uqf/(\Ker \chi_\lambda) = (\UAl)_q$.  We have
\begin{equation}\label{crucial eq}  \Uonef \isoto \U(\g)
\end{equation}
Indeed,
a simple computation, e.g. \cite{BK11}, shows that $\Uqf \cap \C\Lambda = \C 2\Lambda_{+}$ (in fact,  one has $\Uq = \Uqf \ot_{\C 2\Lambda_{+}}
\C\Lambda$).
It is well-known that $\U_1 \cong \U(\g)[K_{\omega_\alpha}; \alpha \in \Delta]/(K^2_{\omega_\alpha} - 1; \alpha \in \Delta)$.
$\U_1 = \U^{\res}_1$ acts adjointly on both sides.
Taking $ad_{\U_1}$-integrable parts \ref{crucial eq} follows.

Consider the extended algebra
$$\tUA := \UAf \ot_{\ZHC_\A} \A
\Lambda.$$
$\A\Lambda$ is free over $\ZHC_\A$ and hence $\tUA$ is free over $\UAf$.
Put $\tU := (\tUA)_q = \Uqf \ot_{\ZHC} \C
\Lambda$.

\subsubsection{}
The sheaf of extended (resp. $\lambda$-twisted) quantum differential operators on the quantum flag
manifold is by definition the $B_q$-equivariant (for the diagonal $B_q$-action) $\Oq$-module $\Dq
:= \Oq \ot \Mq$ (resp. $\Dql := \Oq \ot \Mql$). These are objects in certain categories of quantum
$\D$-modules on the quantum flag manifold, \cite{BK06}. 

We don't recall the definition of these categories here but mention
that the significance of $\Dq$ and $\Dql$ is that they represent the global section functors $\Gamma$ (which in equivariant language is the functor $( \, \;
)^{B_q}$ of taking $B_q$-invariants.) on them. As the usual sheaf of differential operators represents global section
on the category of $\D$-modules this also justify their names.

There is also the induction
functor $\Ind := \IndBG: \mBq \to \mGq$,  see
\cite{APW}. $\mBq$
has enough injectives and so there is the derived functor $R\Ind$.
We have $R\Indd \Mq = R\Gamma(\Dq)$ and $R\Indd \Mql  =
R\Gamma(\Dql)$.

For $\mu \in \Lambda_+$ let $\Ho^0_q(\mu) = \Indd \C_{-\mu}$ be the corresponding dual Weyl module.

Again, these constructions and results have $\A$-versions and we also write $\Ind$ for $\Ind^{G_\A}_{B_\A}$.

\subsubsection{}
By \cite{DP92}, Section 10, $\UA$ has an exhaustive filtration of
finitely generated $\A$-submodules whose associated graded ring is
generated by a finite set of skew-commutative variables. Thus
$\UAf$ has such a filtration as well. It follows that $\UA$,
$\UAf$ and $\tUA$ are noetherian and \emph{generically flat} over
$\A$,  i.e. for any finitely generated module $M$ over one of these
rings there exists $0 \neq f \in \A$ such that $M_f$ is free over
$\A_f$, see \cite{M90}.

\section{Global sections of the sheaf of quantum differential operators on the quantum flag manifold.}\label{Global sections of the sheaf of quantum differential operators on the quantum flag manifold}
\subsection{{}} Here we construct the map $\phi = \phi_\A: \tUA \to \Indd \MA$.
\subsubsection{} The identifications
$$\Indd  \MA = (\OA \ot \MA)^{B_\A} = \End_{\operatorname{Mod}(\D_\A, B_\A, \U_\A(\mathfrak{n}))}(\OA \ot \MA),$$
where $\operatorname{Mod}(\D_\A, B_\A, \U_\A(\mathfrak{n}))$ is the category of  equivariant $\D_\A$-modules defined in 
\cite{BK11}, and the canonical algebra structure on the right hand side defines an algebra structure on
$\Indd \MA$. 

The image of the embedding $\A\Lambda \hookrightarrow \UA(\bobar)  \cong
\MA$
is invariant under the adjoint $B_\A$-action on $\MA$ and induces therefore an algebra embedding
$$
\alpha: \A\Lambda \cong 1 \ot \A\Lambda \hookrightarrow (\OA \ot \MA)^{B_\A} = \Indd \MA.
$$
We shall henceforth consider $\Indd \MA$ as a right $\A\Lambda$-module by means of $\alpha$.

Consider next the composition $$\epsilon: \UAf
\overset{coact.}{\longrightarrow} \Indd  \UAf \to \Indd  \MA.$$ (Note that $coact.$ - given by the tensor identity - is an isomorphism.) Thus we get the map
$\epsilon \ot \alpha: \UA \ot \A \Lambda \to \Indd \MA$. Since $G_\A$ acts
trivially on $\Z^{HC}_\A$ we have $coact.(z) = 1 \ot z$, for $z \in
\Z^{HC}_\A$. Thus, $\epsilon(z) = 1 \ot \chi(z) = \alpha(\chi(z))$.
Thus $\epsilon \ot \alpha$ descends to the desired algebra map $\phi$.

\subsubsection{}
There are two parameter spaces - $T_\Lambda$ and $\C^*$ - that we
shall need to commute with the induction functor:

In Theorem
\ref{higher vanishing of Verma} we shall see that induction (of the modules we are interested in, namely: $\MA, \MAl, \Mq$ and $\Mql$) commutes with $( \ )  \ot_{\A} \C_\lambda$, for $\lambda \in T_\Lambda$, at least if one avoid a finite number of $q$'s.
The argument given there is independent of the preceding results so we may for now simply ignore this subtlety.

\medskip

On the other hand, for  specialization $t \mapsto q  \in \C^*$ we note that since $\MA$ and $\MAl$ are flat $\A$-modules the natural maps
$$
(\Indd \MA)_q \to \Indd \Mq, \  (\Indd \MAl)_q \to \Indd \Mql
$$
are injective. During the proof of Corollary \ref{correct global sections} we shall see that these are isomorphisms.
In the meantime we must strictly speaking distinguish between the various maps
$$
(\phi_\A)_q: \tU \to (\Indd \MA)_q, \  \phi_q: \tU \to \Indd \Mq,
$$
$$
(\pAl)_q: \Uql \to (\Indd \MAl)_q, \  \pql: \Uql \to \Indd \Mql.
$$
Thus, here the rightmost maps
factors through the leftmost maps on respectively row.
\subsection{{}} In this section we state and prove our main results.
%
%
\subsubsection{}\label{correct global sections root of unity}
\Theorem \emph{Let $q$ be a root of unity. Then $\phi_q$ and $\pql$
are isomorphisms for all $\lambda \in T_\Lambda$.}
\begin{proof} \emph{Step a)}  \emph{$\pql$ is injective for any  $\lambda \in T_\lambda$.}
Let  $\ell$ be the order of $q$ (recall that by assumption $\ell$ is odd and $\geq c_\g$).

Let $\uq$ be Lusztig's
small quantum group (see Section \ref{struct  theory sec}). By Proposition \ref{check iso at root}  we
have that $\pql$ restricts to an isomorphism
\begin{equation}\label{to do natural}
(\pql)^{\uq}: \Z(\Uql) =( \Uql)^{\uq} \isoto (\Ind \, \Mql)^{\uq}.
\end{equation}
This way
$\pql$ becomes a morphism of finite $\Z(\Uql)$-algebras. By
\cite{BG01} $\Uql$ is generically Azumaya over
$\Z(\Uql)$, which means that for a sufficiently generic
$\mathfrak{m} \in \mSpec \Z(\Uql)$ we have $\Uql/(\mathfrak{m})
\cong \Mat_N(\C)$. Hence, if $u \in \Ker \pql$, then for such
$\mathfrak{m}$ the image of $u$ vanishes in each fiber
$\Uql/(\mathfrak{m})$, since $\Mat_N(\C)$ is simple.

Since $\Uql$ is finite over $\Z(\Uql)$, we can find $f \in \Z(\Uql)$
such that the localization $(\Uql)_{f}$ is both free and Azumaya
over $\Z(\Uql)_f$. Since $\Uql$ is an integral domain the canonical
map $\Uql \to (\Uql)_f$ is injective. Since $u$ vanishes in all
fibers of $(\Uql)_f$ over $\mSpec \Z(\Uql)_f$, we get from the
Nullstellensatz that $u =0$ and hence that $\pql$ is injective.

\medskip

\noindent \emph{Step b)}  \emph{Let $\lambda \in T_\Lambda$. Then $\pql$ is an isomorphism for
all $q$ for which $\pql$ is injective, in particular for all roots of unity}.  For a $T_q$-module $V$ and $\mu \in \Lambda$ we let $V^\mu$ be the corresponding weight space.
Since $\pql$  is $T_q$-linear with  respect to the restriction to $T_q$
of the adjoint $G_q$-action on $\Uql$ and the $T_q$-action on $\Indd \Mql$ that is induced from the adjoint $T_q$-action on $\Mql$,  it suffices to
show that $(*)$ $a_{q,\mu} \geq b_{q,\mu} < \infty$ where
$$
a_{q,\mu} := \dim (\Uql)^\mu \hbox{ and } b_{q,\mu} := \dim (\Indd \Mql)^\mu, \; \mu \in \Lambda.
$$
Take a $B_q$-filtration on $\Mql$ satisfying $\gr \Mql = \oplus_{\nu \in \Lambda_r} \C^{\dim \Mql^{-\nu}}_{-\nu}$ (recall we consider the adjoint $B_q$-action on $\Mql$). Put
$b'_{q,\mu} := \dim (\Indd \gr \Mql)^\mu$. Then $b_{q,\mu} \leq b'_{q,\mu}$ for all $\mu \in \Lambda_r$.  Moreover, by the Borel-Weyl-Bott theorem $b'_{q,\mu}$ is constant in $q$ for all $q \in \C^*$
(except the usual exceptional roots of unity).

There are $\lambda', \lambda''  \in \h^*$ such that $\MAl \ot_\A \C_1$ equals the classical Verma module $M(\lambda')$
and the  specialization isomorphism $\U^{\fin}_1 \isoto \U(\g)$ induces an isomorphism
$\Uonel \cong \U^{\lambda''}(\g)$.  (Actually $\lambda' = \lambda''$ but we don't need this.)
By classical Beilinson-Bernstein localization, e.g. \cite{Mi}, and by Kostant's formula we have
$$\dim(\U^{\lambda''}(\g))^\mu =  \dim(\Ind^G_B \, M(\lambda'))^\mu =  \dim(\Ind^G_B \, \gr M(\lambda'))^\mu, \hbox{ for all } \mu \in \Lambda$$
Thus we get
$b'_{1,\mu} \leq a_{1, \mu} < \infty$.  Hence,  $a_{q,\mu} \geq b'_{q,\mu}$  by Lemma  \ref{specialSemisimple} ii) and $(*)$ thus follows.
\footnote{Note that we have proved $b_{q,\mu} = b'_{q,\mu}$. This holds despite the fact that $R^{>0}\Indd \gr \Mql \neq 0$ in general.}

\medskip

\noindent \emph{Step c)} It follows from Schur's Lemma that $\phi_q$ is an isomorphism as well. (Compare with Lemma \ref{specialSchur}.)
\end{proof}

\subsubsection{}\label{correct global sections} For any $\A$-algebra $\A'$ we denote by $\phi_{\A'}$ and $\phi^\lambda_{\A'}$ the maps obtained from $\phi_\A$ and $\pAl$ by base change.
We shall here consider only the case when $\A'$ is flat as an $\A$-module. In this case $\Ind$ commutes with base change.
\Cor \emph{There is an $f \in \A \setminus \{0\}$ such that $\phi_{\A_f}$ and  $\phi^\lambda_{\A_f}$ are isomorphisms for all $\lambda \in T_\Lambda$.
Consequently,  $\phi_q$ and $\pql$
are isomorphisms for all $q$ such that $f(q) \neq 0$.}

\begin{proof} Let $q$ be a root of unity. Then $\phi_q$ is injective by Theorem \ref{correct global sections root of unity} and hence also
$(\phi_{\A})_q$ is injective. Since $\Indd \MA$ is flat over $\A$ and $\UA$ is free over $\A$ this implies that $\phi_\A$ is injective as well.

Let $C = \Coker \phi_\A$.  In the proof of Lemma \ref{spectral sequence lemma} we construct a $\UAf$-module structure on $\Indd M$ for each
$B_\A$-equivariant $\UAf$-module $M$. Note that $\phi_\A$ is $\UAf$-linear and that therefore, by Lemma \ref{spectral sequence lemma}, $C$ is a finitely $\UAf$-module as $\Indd \MA$ is.
By generic flatness we can find $f \in \A$ such that $C_f := C \ot_\A \A_f$ is $\A_f$-free.  Let $q$ be a root of unity such that $f(q) \neq 0$. Then, as $\Coker \phi_q = 0$  by Theorem \ref{correct global sections root of unity}, 
we conclude that $(C_f)_q = 0$ and hence that $C_f = 0$.
Thus $\phi_{\A_f}$ is an isomorphism.

It follows from Schur's lemma that $\phi^\lambda_\A$ is an isomorphism for all $\lambda$.  (Compare with $\emph{Step c)}$ in the proof of Theorem \ref{higher vanishing of Verma} and Lemma \ref{specialSchur}.)

It now also follows that $(\phi_\A)_q = \phi_q$ and $(\pAl)_q = \pql$ and that these are isomorphisms for all $q$ with $f(q) \neq 0$ and all $\lambda \in T_\Lambda$.
\end{proof}

\subsubsection{}\rem Corollary \ref{correct global sections} can be partially strengthened to $\phi_q$ is injective for all generic
$q \in \C^*$. Indeed, let $u \in \tU$ and let $f_1 \ot u_2 \in \Oq \ot \tU$ be the
coaction of $\Oq$ on $u$. Thus,  $\langle f_1, v \rangle  u_2 =
ad_r(v)(u)$, for all $v \in \Uqres = \Uq$, where $\langle \ , \
\rangle$ is the pairing between $\Oq$ and $\Uq$. Then $\phi_q(u) =
f_1 \ot \overline{u}_2$ where $\overline{u}_2$ is the image of
$u_2$ in $\Mq$.

If  $u \in \Ker \phi_q$ we thus get $ad_r(\Uq)(u) \in \Uq \cdot \Uq(\n)_{>0}$. Thus $ad_r(\Uq)(u)$ annihilates any highest weight vector in any $\Uq$-module.
Thus $u$ annihilates any finite dimensional simple $\Uq$-module. It is showed in \cite{Jan96} that this in turn implies $u = 0$.

We believe, of course, that $\phi_q$ is an isomorphism for all generic $q$.

\subsubsection{}\label{higher vanishing of Verma} \Theorem
\emph{\textbf{i)} For all $\lambda \in T_\Lambda$ and all roots of unity $q$ we have
$R^{>0}\Indd \Mql = R^{>0}\Indd \Mq =0$ and the natural map
$(\Indd \Mq)_\lambda \to \Indd \Mql$ is an isomorphism.}

\emph{\textbf{ii)}
There is an $f \in \A \setminus \{0\}$ such that the conclusion of \textbf{i)} holds if $q$ is generic and $f(q) \neq 0$.}

\emph{\textbf{iii)}  Similarly, $R^{>0}\Indd \MAfl = R^{>0}\Indd \MAf =0$ and the
natural map $(\Indd \MAf)_\lambda \to \Indd \MAfl$ is an isomorphism
for all $\lambda \in T_\Lambda$.}
\begin{proof}   $\emph{Step a)}$   \emph{$R^{>0}\Indd \Mql = 0$ for all $\lambda \in T_\Lambda$ and all roots of unity $q$.}
Let $V$ be a finite dimensional $G_q$-module. Then it is well-known that $\Mql \ot V$ has a filtration whose subquotients are isomorphic
to $M_q({\lambda+\mu})$ where $\mu$ run through the weights of $V$. By Proposition \ref{check iso at root} we have $(R^{>0}\Indd M_q(\lambda+\mu))^{\uq} = 0$.
By the tensor-identity $R^{>0}\Indd (\Mql \ot V) = (R^{>0}\Indd \Mql) \ot V$. Hence it follows that
 $$((R^{>0}\Indd  \Mql ) \ot V)^{\uq} = 0.$$  It follows that  $R^{>0}\Ind \Mql = 0$.

\medskip

\noindent
$\emph{Step b)}$ \emph{There is an $f \in \A \setminus \{0\}$ such that $R^{>0}\Indd \Mql = 0$ for all $\lambda \in T_\Lambda$ and all generic $q$ such that $f(q) \neq 0$.}
This is done with similar arguments to those of the proof of Corollary \ref{correct global sections}. Details are left to the reader.

\medskip

\noindent $\emph{Step c)}$ \emph{ $R^{>0}\Indd \Mq = 0$ and the natural map $(\Indd \Mq)_\lambda \to \Indd \Mql$ are isomorphisms for all $\lambda$ and all $q$ that are either roots of unity or satisfies $f(q) \neq 0$.}
Chose an identification $\C\Lambda = \C[x^{\pm}_1, \ldots , x^{\pm}_\ell]$.  Let $0 \leq i \leq \ell$. For $\mu = (\mu_1, \ldots , \mu_i) \in (\C^*)^i$ we define $M_q(\mu) := \Mq/\Mq \cdot (x_1-\mu_1, \ldots , x_i - \mu_i)$.
If $i = 0$ we have $M_q(\mu) = \Mq$ and, for $\lambda \in (\C^*)^\ell = \mSpec \C\Lambda$, $\Mql$ is the Verma module with highest weight $\lambda$ as it was previously defined.
Let us make the induction hypothesis
$$
\I_i: \ \ \ \ \  \  \ \ \ \ \ \ \ \ \ \ \ \ \ \ \ \ \  R^{>0}\Indd M_q(\mu)= 0, \ \forall \mu \in (\C^*)^i.  \ \ \  \  \ \ \ \ \ \ \ \ \ \ \ \ \ \ \ \ \  \ \ \  \  \ \ \ \ \ \ \ \ \ \ \ \ \ \ \ \ \
$$
Then $\I_\ell$ holds by \emph{Step b)} and \emph{Step c)}. We show $\I_i \implies \I_{i-1}$. Fix $\mu' \in (\C^*)^{i-1}$ and put $\mu = (\mu',{\mu_i})  \in (\C^*)^i$,
for ${\mu_i} \in \C^*$. We get exact sequence
$$
0 \to M_q(\mu') \overset{x_i - {\mu_i}}{\to} M_q(\mu') \to M_q(\mu) \to 0,
$$
where the first map is right multiplication by  $x_i-{\mu_i}$ which we note is $\Uqf-B_q$-linear. Applying $\Ind$ we get exact sequences
$$
R^{j}\Indd M_q(\mu') \overset{x_i - {\mu_i}}{\to} R^{j}\Indd M_q(\mu') \to R^{j}\Indd M_q(\mu).
$$
For $j \geq 1$ the last term vanishes by hypothesis so $x_i - {\mu_i}$ is surjective. Since, by Lemma \ref{spectral sequence lemma},  $R^{j}\Indd M_q(\mu')$ is a finitely generated
$\Uqf$-module and $\mu_i \in \C^*$ is arbitrary we get from Lemma  \ref{specialSchur} applied to $x = x_i$ and $R = \Uqf$ that $R^{j}\Indd M_q(\mu') = 0$. This proves the induction step. In particular,
$\I_0$ holds, i.e. $R^{>0}\Indd \Mq = 0$.

It now follows that we get exact sequences
$$
\Indd \Mq(\mu') \overset{x_i - {\mu_i}}{\to} \Indd \Mq(\mu') \to \Indd \Mq(\mu) \to 0,
$$
i.e. $\Indd \Mq(\mu) \cong (\Indd \Mq(\mu'))/(x_i-{\mu_i})$. Hence, by an induction starting with $i = 0$ we conclude that
$(\Indd \Mq)_\lambda \overset{def}{=} (\Indd \Mq)/(x_1-\lambda_1, \ldots , x_\ell - \lambda_\ell) \cong \Indd \Mql$, for $\lambda = (\lambda_1 , \ldots , \lambda_\ell) \in (\C^*)^\ell$.

\medskip

\noindent $\emph{Step d)}$ \emph{The assertions about $\A_f$-forms hold.} This is done by applying $\Indd$ and Schur's lemma to
short exact sequences of the form $\MAf \overset{t-q}{\hookrightarrow} \MAf \twoheadrightarrow \Mq$ and $\MAfl \overset{t-q}{\hookrightarrow} \MAfl \twoheadrightarrow \Mql$. Details are left to the reader.
\end{proof}

\subsubsection{}
\rem
For $\lambda \in T_\Lambda$ dominant one can prove that $R^{>0}\Indd \Mql = 0$ for all generic $q$ as follows:
Let $\DBGl$ be the category of $\lambda$-twisted $\D_q$-modules on the quantum flag manifold as defined in \cite{BK11}.
Assume that $\lambda \in T_\Lambda$ is dominant. The quantum version of the celebrated  ``Beilinson-Bernstein trick",
\cite{BK11}, Theorem 5.1 (or \cite{BK06}, Theorem 4.12), then shows that $R^{>0}\Gamma(V) = 0$ for all  $V \in \DBGl$. In particular we
get
$$
R^{>0}\Ind \Mql = R^{>0}\Gamma(\Dql) = 0,  \hbox{ for } \lambda \hbox{ dominant,}
$$
where  $\Dql := \Oq \ot \Mql$.

This can most likely be extended to all $\lambda$'s by quantizing the intertwining functors of Beilinson and Bernstein, \cite{BB83}, and the derived equivalences they define.  We haven't worked out the details.

\section{Applications to structure theory for the quantum group}\label{struct  theory sec}
\subsection{{}}
\subsubsection{}
The following quantum version of a classic result of Kostant,
\cite{K63}, was originally proved in \cite{JL94}, for a generic
$q$. A nice proof that used Kashiwara's crystal bases was 
given in \cite{B00}, also that in the generic case. The proof
given here works for all roots of unity $q$ and all generic $q$ except perhaps a finite number of algebraic numbers.
\begin{cor} \textbf{i)} (Separation of variables.) There is a filtration
of the form $L_i \ot \ZHC$ on $\Uqf$ for some $G_q$-submodules
$L_i \subset \Uqf$ and a subspace $\mathcal{H}_q \subset \gr \Uqf$
such that multiplication gives an isomorphism $ \mathcal{H}_q \ot
\ZHC \isoto \gr \Uqf$. In particular, $\Uqf$ is free over $\ZHC$.
\textbf{ii)} $\mathcal{H}_q$ is a direct sum of
$\Ho^0_q(\lambda)$'s. The multiplicity of $\Ho^0_q(\lambda)$ in this
sum equals the dimension of the $0$'th weight space  $\Ho^0_q(\lambda)^0 \subset \Ho^0_q(\lambda)$.
\end{cor}
\begin{proof}
Let
$m_\mu = \dim \Uq(\nbar)^{-\mu}$. Consider a vector space basis
$v_0, v_1, v_2, \ldots $ of $\Uq(\nbar)$ with the property that
$v_i$ is a weight vector of weight $-\mu_i$ and $\mu_i
> \mu_j \implies i > j$.  Define a filtration $F$ on $\Uq(\bobar) \cong \Mq$ by
$F_i := \operatorname{Span}_\C\{v_0, v_1, \ldots , v_i\} \ot
\C\Lambda$. Then
$$
 \gr_F
\Mq \cong (\oplus_{\mu \in (\Lambda_r)_{+}} \C^{m_\mu}_{-\mu}) \ot
\C\Lambda,
$$
$$
\Ind \gr_F \Mq = (\oplus_{\mu \in (\Lambda_r)_{+}} \Ho^0_q(\mu)^{m_\mu})
\ot \C\Lambda.
$$
(Of course,  $\Ho^0_q(\mu) = 0$ unless $\mu \in \Lambda_+$.) Consider the induced filtration $\Ind F_i$ on $\tU \cong \Indd \Mq$. We get the injection $$\gr_{\Ind F} \tU \cong \gr_{\Ind F} \Ind \Mq \overset{nat}{\to}  \Ind \gr_F \Mq.$$
It follows from the multiplicity computations in the proof of Theorem \ref{correct global sections} that $nat$ is an isomorphism.
Thus we get an isomorphism $$\psi_q: \gr_{\Ind F} \tU \cong (\oplus_{\mu
\in (\Lambda_r)_+} \Ho^0_q(\mu)^{m_\mu}) \ot \C\Lambda.$$ We define
${\mathcal{H}}_q :=  \psi^{-1}_q(\oplus_{\mu \in  (\Lambda_r)_+}
\Ho^0_q(\mu)^{m_\mu})$. $\tW$ acts naturally on $\tU$ and we see that the $\tW$-invariant subspace of each step in the filtration is of the form
${(\Ind F)}^{\tW}_i = L_i \ot \ZHC$, for some $G_q$-invariant subspace $L_i \subset \Uq$.
This proves \textbf{i)}. An application of Weyl's character
formula shows that $m_\lambda = \dim \Ho^0_q(\lambda)^0$, which
proves \textbf{ii).}
\end{proof}
Note that for $q$ generic we have $\Uqf \cong \gr \Uqf$, as
$G_q$-modules, and we may replace the statement of \textbf{i)} by
``${\mathcal{H}}_q \subset \Uqf$ is a $G_q$-submodule such that
${\mathcal{H}}_q \ot \ZHC \cong \Uqf$". This is how \cite{JL94,
B00} and classically \cite{K63} stated the result.

\begin{cor} (Duflo's formula.) Assume that $q$ is generic and let $\lambda \in T_\Lambda$.
Then $\Uql$ acts faithfully on $\Mql$.
\end{cor}
\begin{proof} Let $J = \Ann_{\Uql} \Mql$. Observe that $J$ is a two-sided ideal in $\Uql$.
Since $q$ is generic we have $\Uqres = \Uq$ so that $J$ is
actually stable under $ad(\Uqres)$. Let $u \in J$. We have
$coact.(u) = u_1 \ot u_2 \in \Oq \ot \Uql$, where $\langle u_1, v
\rangle  u_2 = ad(v)(u)$, for all $v \in \Uqres$. Here $\langle
\ , \ \rangle$ is the pairing between $\Oq$ and $\Uqres$. We have
$\pql(u) = u_1 \ot \overline{u}_2$ where
$\overline{u}_2$ is the image of $u_2$ in $\Mql$. Thus, $\pql(u) = 0$ and since $\pql$ is injective we
conclude that $u=0$.
\end{proof}
\begin{rem} Here is Duflo's formula for $q$ an $\ell$'th root of unity:
For \emph{unramified} $\mathfrak{m}_\tau \in \mSpec \Z(\Uqf)$ it
is known that $\Uqf/(\mathfrak{m}_\tau) \isoto \End(M_\tau)$,
where $M_\tau$ is the corresponding baby Verma module (see
\cite{BG01}). From this one deduces that
$$\Ann_{\Uqf}(M_\lambda) =
\Uqf \cdot \{\Ker \chi_\lambda, E^{\ell}_\alpha, K_{\ell\mu} -
\lambda(K_{\ell\mu}), \alpha \in \Delta, \mu \in \Lambda\}.$$
\end{rem}

\subsection{Centers at a root of unity.}\label{centers}
\subsubsection{}
Let $q$ be an $\ell$'th root of unity. Let $\uq \subset \Uqres$ be
Lusztig's small quantum group and let $V^{\uq}$ be the set of
$\uq$-invariants in a $\Uqres$-module $V$. Since $\uq$ is the
algebra kernel of the quantum Frobenius map $Fr: \Uqres \to
\U(\g)$ we see that $V^{\uq}$ has a $\g$-action.

Let $\Z = \Z(\Uq)$ be the center of $\Uq$; then $\Z = \Uq^{\uq}$
and, hence, $\Z$ has a $\g$-action, again denoted $ad$ (which is
trivial on $\ZHC$). Let $$\Zl := \C \langle E^{\ell}_\alpha,
F^{\ell}_\alpha, K_{\ell\gamma}; \alpha \in \Delta, \gamma \in \Lambda
\rangle \subset \Z$$ be the $\ell$-center of $\Uq$ and put
$\Z^{(\ell)}_0 := \C \langle E^{\ell}_\alpha, F^{\ell}_\alpha, K_{\ell\gamma};
\alpha \in \Delta, \gamma \in 2\Lambda \rangle \subset \Zl$. Then
$\Z^{(\ell)}_0$ and $\Zl$ are $\g$-module subalgebras of $\Z$ and,
moreover, $\Zl$ is free of rank $2^{\rank \g}$ over $\Z^{(\ell)}_0$ with basis
$J := \{K_\gamma; \gamma = \sum_{ \alpha \in \Delta }
\epsilon_\alpha \ell \omega_{\alpha},\; \epsilon_\alpha \in \{0,1\}
\}$. Let $G_0 \subset G$ be the open Bruhat cell and let $T \subset G_0$ be the torus. Consider the
adjoint $\g$-action on $\BGG(G_0)$.  We have
\subsubsection{}\label{first}
\Prop \emph{The $\g$-module algebras $\BGG(G_0)$ and
$\Z^{(\ell)}_0$ are isomorphic.}
\begin{proof}  In \cite{CKP92}, Theorem 5.5, the commutator action
of $\Uqres$ on $\Z^{(\ell)}_0$ was explicitly calculated. This action and the
adjoint action that we favour are closely related: For each $\alpha \in \Delta$ one constructs $\gamma_\alpha \in
\Lambda$ such that $ad_{E^{(\ell)}_\alpha}( \ ) = K_{\gamma_\alpha}
[E^{(\ell)}_\alpha,\ ]$ and similarly for the ${F^{(\ell)}_\alpha}{}'s$.
The proposition follows from this and the computations in \cite{CKP92}.
\end{proof}

\subsubsection{}\label{cover adios}
Hence we get a $\g$-module algebra inclusion $\BGG(G_0)
\hookrightarrow \Zl$; it becomes an equality after taking
$\g$-integrable parts
\begin{proposition}
$\BGG(G)
= \BGG(G_0)^\fin \cong \Zl(\Uqf) := \Zl \cap \Uqf.$
\end{proposition}
\begin{proof} The first equality holds since $G_0$ is open and dense in $G$ and $ad_G(G_0) = G$.
Next we have the $\g$-module decomposition $$\Zl = \oplus_{\gamma \in J}
\, \BGG(G_0)\cdot K_\gamma.$$
We must show that $(\BGG(G_0)\cdot
K_\gamma)^{\fin} \neq 0 \implies \gamma = 0$, for $\gamma \in J$.
Let $0 \neq f \in \BGG(G_0)$ be such that $f \cdot K_\gamma \in
(\BGG(G_0) \cdot K_\gamma)^{\fin}$. By multiplying $f$ with a
suitable invertible element of $\BGG(T)$,  we can assume that $f \in \BGG(G)$.

Note that under the isomorphism of Proposition \ref{first},
$\BGG(T) \cong \C 2\ell \Lambda$ as $\g$-module algebras. If $\gamma
\neq 0$, we may pick $\alpha \in \Delta$ such that $\langle
\alpha, \gamma \rangle \neq 0$, and it follows readily that
$ad^n(E_\alpha)(f \cdot K_\gamma) \neq 0$, for all $n  \geq 0$,
which contradicts the integrability of $ad(E_\alpha)$ on $f \cdot
K_\gamma$. Hence, $\gamma = 0$.
\end{proof}

\subsubsection{}\label{no cover implication}
The composition $G_0 \hookrightarrow G \to G // G = T/ \W$ gives the
inclusions ${\Zl \cap \ZHC} \cong {{\C \ell\Lambda}^{\tW}} = \BGG(T)^\W \subset \BGG(G_0) \subset \Zl$
and it is known that $ \Z = \Zl \ot_{\Zl \cap \ZHC} \ZHC \cong \Zl
\ot_{{\C \ell\Lambda}^{\tW}} \C \Lambda^{\tW}$. Hence Proposition
\ref{cover adios} gives
\begin{cor}
$\Z(\Uqf) = (\Uqf)^{\uq} \cong \BGG(G) \ot_{{\C \ell\Lambda}^{\tW}} \C
\Lambda^{\tW}, \ \, \Z(\tU) = \tU^{\uq} \cong \BGG(G) \ot_{{\C
\ell\Lambda}^{\tW}} \C \Lambda.$
\end{cor}

Next we have
\subsubsection{}\label{check iso at root}
\Prop \emph{The isomorphisms $\phi_q$ and $\pql$ of Theorem \ref{correct global sections} restricts to isomorphisms
$\phi^{\uq}_q: \tU^{\uq} \isoto (R\Indd  \Mq)^{\uq}$ and $(\pql)^{\uq}: (\Uql)^{\uq} =  \Z(\Uql)  \isoto (R\Indd  \Mql)^{\uq}$.}

\begin{proof} This follows of course by applying $\uq$-invariants to our main results from Section \ref{Global sections of the sheaf of quantum differential operators on the quantum flag manifold}.
However, this proposition is used to prove those results so we must give an independent proof.

Let $\bq := \Uqres(\bo) \cap \uq$. By Corollary \ref{no cover implication}, $\tU^{\uq}  \cong \BGG(G) \ot_{{\C
\ell\Lambda}^{\tW}} \C \Lambda$ and it follows that $$\Mq^{\bq} = \Jm\{\Uq^{\uq} \hookrightarrow
\Uq \to \Mq\}  \cong  \BGG(B) \ot_{\C 2\ell \Lambda} \C \Lambda,$$ with
the adjoint $B$-action on $\BGG(B)$ (cf \cite{BK08}). 

For any $B_q$-module $V$,  $V^{\bq}$ is a $B$-module, and by \cite{GK93} there  is a natural isomorphism
$R\Ind^G_B \, V^{\bq} \isoto (R\Ind^{G_q}_{B_q} \, V)^{\uq}$.

Thus, in particular
$$(R\Indd  \Mq)^{\uq} \cong R\Ind^G_B \, \Mq^{\bq} \cong (R\Ind^G_B  \, \BGG(B)) \ot_{\C 2\ell \Lambda}
\C \Lambda \cong$$ 
$$(\BGG(G) \ot_{\C \ell\Lambda^{\tW}} \C 2\ell\Lambda)
\ot_{\C 2\ell\Lambda} \C \Lambda = \BGG(G) \ot_{{\C \ell\Lambda}^{\tW}}
\C \Lambda,$$ where the third isomorphism is proved e.g. in
\cite{S82}. Following the maps we conclude that $\phi^{\uq}_q$ is
an isomorphism. That $(\pql)^{\uq}$ is an isomorphism is proved similarly.
\end{proof}
We finish with
\subsubsection{}
\Prop \emph{$\Uqf$ is a free $\BGG(G)$-module of rank   $\ell^{\dim \g}$ and $\tU \cong \Indd \Mq$ are free $\BGG(G)$-modules of rank $| \tW | \cdot \ell^{\dim \g}$.}

\begin{proof}
Since $\tU$ is free over $\Uqf$ of rank $| \tW |$ it suffices to prove the first assertion.
By \cite{DP92}, $\Uq$ is free over $\Zl(\Uq)$ of rank  $\ell^{\dim
\g}$. Let $\eta$ be the composition $\mSpec \Zl \to G_0
\hookrightarrow G = \mSpec \Zl(\Uqf)$. It is clear that for $\m
\in \Jm \eta = G_0$ we have $\Uqf/(\m) \cong \Uq/(\m)$ and hence
$\dim \Uqf/(\m) = \ell^{\dim \g}$.  Since any $\m \in G$
can be moved into $G_0$ by the adjoint $G$-action and $\Uqf$ is
a $G_q$-equivariant $\BGG(G)$-module, it follows that $\dim
\Uqf/(\m) = \ell^{\dim \g}$, for all $\m \in G$. Since $\BGG(G)$ is reduced this implies that $\Uqf$ is projective over $\BGG(G)$, e.g. \cite{H77}.
Since projective $\BGG(G)$-modules of rank $> \dim \g$ are free, see \cite{BG02, MR88},  we are done.
\end{proof}

\section{Appendix}
\subsection{{}} Here we collect some basic facts that we need.

\subsubsection{}\label{specialSemisimple}
\Lemma \emph{Let $V,W$ be $G_\A$-modules and let $\Theta_q: \Hom_{G_\A}(V,W) _q \to \Hom_{G_q}(V_q,W_q)$ be the natural map, for $q \in \C^*$.}

\emph{\textbf{i)} If $W$ is $\A$-flat then $\Theta_q$ is
injective. If moreover $V$ is $\A$-free of finite rank and $q$ is
generic, then $\Theta_q$ is an isomorphism.}

\emph{\textbf{ii)} Assume that $W$ is
$\A$-free. Then  $ \dim (W_q)^\mu =  \dim (W_1)^\mu$ for generic
$q$ and $ \dim (W_q)^\mu \geq \dim (W_1)^\mu$ for all roots of
unity $q$.}

\begin{proof}  For $f \in  \Hom_{G_\A}(V,W)$ denote by $f_q$ its image in
$\Hom_{G_\A}(V,W) _q$. If $\Theta_q(f_q) = 0$, we have by definition $f(V) \subseteq (t-q)W$. Then, since $W$ is $\A$-flat = $\A$-torsion free, we can define
 $f' \in \Hom_{G_A}(V,W)$ by $f' := (t-q)^{-1}f$. Thus $f = (t-q)f'$, so we get $f_q = 0$; hence $\Theta_q$ is injective.

\smallskip

Differentiation identifies $G_\A$-modules with a full subcategory of
$\UA\MOD$ and, for $q$ generic, $G_q$-modules with a full subcategory
of $\Uq\MOD$. Thus, it suffices to show that the natural map
$\Theta'_q: \Hom_{\UA}(V,W) _q \to \Hom_{\Uq}(V_q,W_q)$ is an
isomorphism.  Both functors $\Hom_{\UA}(V, \  ) _q$ and
$\Hom_{\Uq}(V_q, \  )$ commute with direct limits; hence we can
assume that $W$ is finitely generated over $\A$.

The short exact sequence $W \overset{t-q}{\hookrightarrow} W
\overset{\pi}{\twoheadrightarrow} W_q$ yields an exact sequence
$$
(*) \  \  \Hom_{\UA}(V, W)  \overset{\pi_*}{\to} \Hom_{\UA}(V, W_q) \to
\Ext^1_{\UA}(V, W) \overset{t-q}{\to} \Ext^1_{\UA}(V, W)  \to
\Ext^1_{\UA}(V, W_q).
$$
Under the natural identification $ \Hom_{\UA}(V_q, W_q)   = \Hom_{\UA}(V, W_q)$ we have  $\Theta'_q(f_q) = \pi_*(f) $, for $f \in \Hom_{\UA}(V, W)$; thus,
$\Coker \Theta'_q =
\Ann_{\Ext^1_{\UA} (V,W)}(t-q)$. Let $P_\bullet \to V$ be a free
resolution of $V$ in $\UA\text{-}\mood$. Clearly, $P_\bullet$
splits in $\A\text{-}\mood$; therefore $( P_\bullet)_q \to V_q$ is
a free resolution in $\Uq\text{-}\mood$ and so $\Ext^1_{\UA}(V,
W_q)  \cong \Ext^1_{\Uq}(V_q,W_q).$

This group vanishes since the category of finite dimensional $\Uq$-modules is semi-simple.
Thus,  $t-q$ is surjective  on $\Ext^1_{\UA}(V, W)$.  Since  $V$ and $W$ are f.g. over $\A$ it follows that $\Ext^1_{\UA}(V, W)$ is f.g. over $\A$ and therefore multiplication by $t-q$
must also be injective on  $\Ext^1_{\UA}(V, W)$.  Thus $\Coker \Theta_q = 0$. This proves \textbf{i)}.

\medskip

\noindent
\textbf{ii)} As the functors $(( \ )_q)^{\mu} = \Hom_{T_q}(\C_{\mu}, \ )$ and $((  \ )^{\mu})_q = \Hom_{T_\A}(\A_\mu, \ )_q$ commute with direct limits we may assume that $W$ is finitely generated. (Here $\C_{q,\mu}$ and $\A_\mu$ are the rank one representations defined by $\mu$.) Since $W^\mu$ is free over $\A$ we see that $m := \dim_\C (W^\mu)_q$ is constant in $q$, for all $q$.
The argument that we used to establish injectivity of $\Theta_q$ again shows that
$$m \leq \dim (W_q)^\mu, \hbox{ for all } q.$$

Since $\Ext^1_{T_\A}(\A_\mu, W)$ is finitely generated over $\A$ it is countably generated over $\C$ which implies that the number of $q$ such that $t-q$ is not injective on it is at most countable. Let $q$ be a non-root of unity such that $t-q$ is injective on $\Ext^1_{T_\A}(\A_\mu, W)$; then a long exact sequence analogous to $(*)$ shows $m = \dim (W_q)^\mu$.

Hence, if we can prove the first assertion we get $m = \dim (W_1)^\mu$ and the second assertion will thus follows.

For generic $q$ we have
$$
W_q = \oplus_{\nu \in \Lambda_+} (\Ho^0_q(\nu))^{m_{q,\nu}},
$$
where $m_{q,\nu} = \dim \Hom_{G_q}(\Ho^0_q(\nu), W_q)$.  But by \textbf{i)} we then get $m_{q,\nu} = \dim \Hom_{G_\A}(\Ho^0_\A(\nu), W)_q$ which is independent of $q$.
As the numbers $\dim (\Ho^0_q(\nu))^\mu$ are independent of $q$ (indeed, they are given by Weyl's formula), we have proved \textbf{ii)}.
\end{proof}
\subsubsection{}\label{specialSchur}
\Lemma \emph{Let $R$ be a noetherian $\C$-algebra such that $\dim_\C R$ is countable and let  $M$ be a
finitely generated left $R$-module. Let $x \in \Aut_R(M)$ and assume that $(x-a)M =  M$ for all $a \in \C^*$.
Then $M = 0$.}

\begin{proof} Assume that $M \neq 0$.  The hypothesis implies that $\dim_\C M$ is at most countable. Since $\C$ is uncountable and algebraically closed, Schur's lemma implies there
is a $a \in \C$ such that $x - a$ is not invertible on $M$. Then $a \neq 0$ since $x$ is invertible on $M$.
Since any surjective endomorphism of a noetherian
object must be injective, we conclude that $x-a$
is not surjective on $M$. This gives the contradiction.
\end{proof}

\subsubsection{}\label{spectral sequence lemma}
\Lemma \emph{Let $M$ be a noetherian object in the category of $B_\A$-equivariant $\UAf$-modules. Then $R^{>0}\Indd M$ is f. g. over
$\UAf$.}
\begin{proof}
The
$\UAf$-module structure on $\Indd M$ is given by the
composition $\UAf \ot \Indd \MA \cong \Ind (\UAf \ot \MAl) \to \Indd
\MAl$.

The noetherian assumption  is equivalent to $M$ being f.g. $\UAf$-module.
Using that $\UAf$ is
noetherian it is easy to inductively construct a (possibly
infinite) resolution $F_\bullet \to M$ in the category of
$B_\A$-equivariant $\UAf$-modules, where $F_i = \UAf \ot V_i$ and
$V_i$ is a $B_\A$-module which is free of finite rank over $\A$.

We have $R\Indd \UAf \ot V_i = \UAf \ot R\Indd V_i$, by the tensor
identity, and each $R^j\Indd V_i$ is a finitely generated
$\A$-module. Take injective resolutions of each $F_i$ and apply
$\Ind$ to the corresponding double complex. This gives a spectral
sequence whose $E_2$-terms are given by $E^2_{p,q} = \UAf \ot R^p
\Indd V_q$. Moreover, since $\Ind$ has finite cohomological
dimension there are only finitely many non-zero
$E^2_{p+j,p}$-terms for fixed $j \geq 0$. Since the $E^\infty_{p+j,p}$-terms
occur as $\gr$ of a filtration on $R^j\Indd M$ we are done.
\end{proof}

\end{document}